\documentclass[a4paper]{article}

\usepackage{amsmath}
\usepackage{amssymb}
\usepackage{amsthm}
\usepackage{mathtools}
\usepackage{bm}

\numberwithin{equation}{section}

\theoremstyle{plain}
\newtheorem{thm}{Theorem}[section]
\newtheorem{prop}[thm]{Proposition}
\newtheorem{cor}[thm]{Corollary}
\newtheorem{lem}[thm]{Lemma}
\theoremstyle{definition}
\newtheorem{defn}[thm]{Definition}
\newtheorem{rem}[thm]{Remark}

\title{Shapes of light: singularities of slant functions and the Gauss map of a surface\\[4pt]
}
\author{%
  Shuhei\,Honda\thanks{Saitama, Japan.
   \texttt{shuhei.honda.836@gmail.com}}%
  \and
  Yutaro\,Kabata\thanks{%
    Graduate School of Science and Engineering,
    Kagoshima University,
    1-21-35 Korimoto, Kagoshima 890-0065, Japan.
    \texttt{kabata@sci.kagoshima-u.ac.jp}}%
}

\begin{document}

\maketitle

\begin{abstract}
The slant function on a regular surface in three-space, introduced by Koenderink, is the inner product of a fixed light direction with the unit normal; under the standard Lambertian model it gives the brightness of the surface illuminated from that direction. We study the singularities of this function and show that their type is characterised by the differential geometry of the parabolic set and its spherical image under the Gauss map. At every point of the surface, we first describe the light directions along which the slant function has a degenerate singularity; such singularities arise only at the parabolic points. We then determine the precise singularity type at representative parabolic points. At the first-order inflection and first-order vertex of Gauss, and at the cusp of Gauss at which the parabolic set is regular, the finer geometry of the spherical image is reflected in the singularity of the slant function. 
\end{abstract}

\noindent
\textbf{Keywords.}
Slant function; parabolic points; Gauss map; 
singularity theory;
cusp of Gauss; inflection of Gauss; vertex of Gauss.

\medskip
\noindent
\textbf{Mathematics Subject Classification (2020).}
Primary~57R45;
Secondary ~53A05, ~58K05.

\section{Introduction}
\label{sec:intro}

In this paper we study the \emph{slant function} on a regular surface
$\mathcal{M}\subset\mathbb{R}^{3}$, following the terminology of
Koenderink \cite{Koen1990}.
For a fixed unit vector $\mathbf{v}\in S^{2}$, called the \emph{light direction},
the slant function $I_{\mathbf{v}}\colon\mathcal{M}\to\mathbb{R}$ is defined by
\[
    I_{\mathbf{v}}(p) \;=\; \langle \mathbf{v},\,N(p)\rangle,
\]
where $N\colon\mathcal{M}\to S^{2}$ denotes the \emph{Gauss map}.
Although the definition is elementary, the singularities of $I_{\mathbf{v}}$
capture the local geometry of $\mathcal{M}$ and of the Gauss map.
The slant function also has a natural interpretation in computer vision: under the standard Lambertian
reflectance model \cite{Horn1975}, $I_{\mathbf{v}}$ models the
brightness distribution on $\mathcal{M}$ illuminated from direction
$\mathbf{v}$, and its level curves (isophotes) are fundamental
in applications such as shape reconstruction and stylised
rendering \cite{Todo2007}.
Koenderink \cite{Koen1990} introduced the function in this context
and observed that its level curves exhibit qualitatively distinct
patterns depending on the geometry of the surface---in particular,
on the location of the parabolic set---suggesting their significance
for the perception of solid shape.
Koenderink and van Doorn \cite{KvD1993} subsequently analysed the
critical points of the slant function via Hessian analysis, again
highlighting the role of the parabolic set.
However, the non-Morse singularities of the slant function have not
been systematically classified by $A_k$-type.
In the present paper, we provide such a classification and determine
the precise relationship between the singularity type of $I_{\mathbf{v}}$
and the differential geometry of the Gauss map.

It is well known that the parabolic set $\Sigma\subset\mathcal{M}$,
on which the Gauss curvature $K$ vanishes, coincides with the singular
set of $N$, and that the geometry of its spherical image
$\Gamma = N(\Sigma)\subset S^{2}$ reflects finer geometric invariants
of $\mathcal{M}$ beyond the Gauss curvature itself
\cite{ BGM1982,BG1992, BT2019, CG2000, DGH2016, HKS2024, IRFT}.
One of the central points of the present paper is that
the singularity type of $I_{\mathbf{v}}$ detects precisely these invariants.

To state the stratification, at a parabolic point $p$ at
which $\Sigma$ is regular, we let $\gamma$ denote a local arc-length
parametrisation of $\Sigma$ near $p$, and $\kappa$ denote the
geodesic curvature of the spherical curve $N\circ\gamma$ at $p$.

\begin{itemize}
    \item A parabolic point $p$ is called a \emph{fold of Gauss} if
    the restriction of $N$ to $\Sigma$ is regular at $p$;
    equivalently, the height function along $N(p)$ has an
    $A_{2}$-singularity.
    A parabolic point that is not a fold of Gauss is called a
    \emph{degenerate singularity of Gauss}; a classical example
    is the \emph{cusp of Gauss} \cite{IRFT}, at which
    the height function along $N(p)$ has an
    $A_{3}$-singularity.

    \item Among folds of Gauss, $p$ is an \emph{inflection of Gauss}
    if $\kappa(p)=0$, and a \emph{first-order inflection of Gauss}
    if in addition $\kappa'(p)\neq 0$.

    \item Similarly, $p$ is a \emph{vertex of Gauss} if
    $\kappa(p)\neq 0$ and $\kappa'(p)=0$, and a
    \emph{first-order vertex of Gauss} if in addition
    $\kappa''(p)\neq 0$.
\end{itemize}

\begin{rem}
By definition, a cusp of Gauss is a parabolic point at which the
height function along $N(p)$ has an $A_{3}$-singularity.
Note that this condition alone does not imply that the Gauss map
$N$ itself has a cusp singularity in the sense of
$\mathcal{A}$-equivalence \cite{HKS2024, IRFT}.
We use \emph{cusp of Gauss} in its conventional sense.
\end{rem}

The spherical image $\Gamma = N(\Sigma)$ has been studied extensively
from the viewpoint of singularity theory
\cite{BGM1982, BG1992, BT2019, HKS2024, IRFT, PST2024},
and the geometric meaning of the strata above is reflected in $\Gamma$:
the cusp of Gauss corresponds to a singularity of $\Gamma$,
the inflection of Gauss to a geodesic inflection of $\Gamma$,
and the vertex of Gauss to a vertex of $\Gamma$, that is, a critical
point of its geodesic curvature.
The cusp of Gauss has long been a classical object in singularity theory
\cite{BGM1982, BG1992, IRFT}, and the inflection of Gauss has
appeared in connection with certain projection singularities
\cite{BGM1982, HKS2024, IRFT};
see Remark~\ref{rem:literature-parabolic} for further discussion.
The vertex of Gauss, by contrast, is introduced here for the first
time.
Our main results show that these four geometric types are detected
exactly by the $A_{\ge 2}$-directions of the slant function.

\begin{defn}
For $p\in\mathcal{M}$, we call $\mathbf{v}\in S^{2}$ an
\emph{$A_{k}$-direction} at $p$ if $I_{\mathbf{v}}$ has an
$A_{k}$-singularity at $p$.
We call $\mathbf{v}$ an \emph{$A_{\ge 2}$-direction} at $p$ if $I_{\mathbf{v}}$ has
a non-Morse singularity at $p$.
\end{defn}

The existence and types of $A_{\ge 2}$-directions are dictated by
the differential-geometric type of the point on $\mathcal{M}$.
To describe them, for a parabolic point $p\in\mathcal{M}$ we call
the great circle on $S^{2}$ lying in the plane spanned by $N(p)$
and the asymptotic direction at $p$ the \emph{asymptotic normal
section} at $p$;
it coincides with the set of directions
$\mathbf{v}\in S^{2}$ along which $I_{\mathbf{v}}$ has a critical
point at $p$ (a fact established in
Proposition~\ref{prop:critical-directions}).
\begin{thm}
\label{thm:stratification}
Let $p\in\mathcal{M}$.
The $A_{\ge 2}$-directions and $A_{1}$-directions of the slant
function at $p$ are described as follows.
\begin{enumerate}
    \item \textbf{Elliptic or hyperbolic point}:
    No $A_{\ge 2}$-direction exists, and the only $A_{1}$-direction
    is the normal direction $N(p)$ (of type $A_{1}^{+}$).

    \item \textbf{Fold of Gauss}:
    There exist exactly two $A_{\ge 2}$-directions, one of which is
    the normal direction $N(p)$ (of type $A_{3}$).
    The $A_{1}$-directions are precisely the remaining directions
    on the asymptotic normal section at $p$.

    \item \textbf{Degenerate singularity of Gauss with $\Sigma$ regular at $p$}:
    The normal direction $N(p)$ is the unique $A_{\ge 2}$-direction
    (of type $A_{\ge 5}$).
    The $A_{1}$-directions are precisely all remaining directions
    on the asymptotic normal section at $p$.

    \item \textbf{Degenerate singularity of Gauss with $\Sigma$ singular at $p$}:
    Every direction on the asymptotic normal section at $p$ is an
    $A_{\ge 2}$-direction;
    no $A_{1}$-direction exists on the asymptotic normal section at $p$.

    \item \textbf{Flat umbilic point}:
    $I_{\mathbf{v}}$ has a singularity at $p$ for every direction
    $\mathbf{v}\in S^{2}$.
\end{enumerate}
\end{thm}

\begin{rem}
The coarser distinction between Morse and non-Morse directions
in the above theorem was established in \cite{KvD1993} via Hessian
analysis. The present paper refines this by classifying each
non-Morse singularity by its precise $A_k$-type, yielding a
finer stratification.
\end{rem}

Our second main theorem completely determines the singularity type
of $I_{\mathbf{v}}$ at each $A_{\ge 2}$-direction, for the four first-order
representative cases of parabolic points.
In each case, the geometric order of the point---expressed through
the order of vanishing of the relevant invariants---fixes the
singularity type uniquely.

\begin{thm}
\label{thm:parabolic}
Let $p$ be a parabolic point of $\mathcal{M}$ that is not a flat umbilic.
\begin{enumerate}
    \item If $p$ is a fold of Gauss but neither an inflection nor
    a vertex of Gauss,
    then the $A_{\ge 2}$-directions are exactly:
    \begin{itemize}
        \item the normal direction\,: type $A_{3}$,
        \item a second $A_{\ge 2}$-direction, other than the asymptotic
              direction\,: type $A_{2}$.
    \end{itemize}

    \item If $p$ is a first-order inflection of Gauss,
    then the $A_{\ge 2}$-directions are exactly:
    \begin{itemize}
        \item the normal direction\,: type $A_{3}$,
        \item the asymptotic direction\,: type $A_{2}$.
    \end{itemize}

    \item If $p$ is a first-order vertex of Gauss,
    then the $A_{\ge 2}$-directions are exactly:
    \begin{itemize}
        \item the normal direction\,: type $A_{3}$,
        \item a second $A_{\ge 2}$-direction\,: type $A_{3}$.
    \end{itemize}

    \item If $p$ is a cusp of Gauss at which $\Sigma$ is regular,
    then the normal direction is the unique $A_{\ge 2}$-direction,
    of type $A_{5}$.
\end{enumerate}
\end{thm}

\begin{rem}
The phenomena above show that the slant function $I_{\mathbf{v}}$,
despite its formal resemblance to the height function
$h_{\mathbf{v}}(p) = \langle \mathbf{v}, p\rangle$, behaves
quite differently.
Both are scalar functions on $\mathcal{M}$ defined by an inner product
with a fixed vector $\mathbf{v}$, but they capture different geometry:
$h_{\mathbf{v}}$ measures the height of $p$ along $\mathbf{v}$,
whereas $I_{\mathbf{v}}(p) = \langle\mathbf{v},N(p)\rangle$ measures
the cosine of the angle between $N(p)$ and $\mathbf{v}$.
The singularity theory of height functions on surfaces is a
classical subject \cite{BGM1982,IRFT}.
On a generic surface they admit at most an $A_{3}$-singularity,
and yield a stratification into elliptic ($A_1^+$), hyperbolic ($A_1^-$),
fold-of-Gauss ($A_2$), and cusp-of-Gauss ($A_3$) points.
By contrast, $I_{\mathbf{v}}$ admits a finer stratification
controlled by higher-order invariants of $\Sigma$ and of its
spherical image $\Gamma$, as Theorem~\ref{thm:parabolic} makes precise.
We also remark that, at the normal direction $N(p)$, as the geometric
type of $p$ degenerates from elliptic/hyperbolic through a fold of Gauss
to a cusp of Gauss, only the odd-index singularities $A_{1}$, $A_{3}$,
$A_{5}$ arise; the even-index types $A_{2}$ and $A_{4}$ never occur
in this family.
\end{rem}


In Section~2 we review the necessary background from differential
geometry and singularity theory.
In Section~3 we derive algebraic criteria for the singularity types
of $I_{\mathbf{v}}$ in Monge form.
In Section~4 we translate these criteria into geometric invariants
of the Gauss map and prove the two main theorems.

\section{Preliminaries}
\label{sec:prelim}

Standard references are \cite{BG1992, IRFT, Por2001} for differential
geometry and \cite{BG1992, IRFT} for singularity theory.

\subsection{Surfaces in $\mathbb{R}^{3}$ and the Gauss map}
\label{subsec:prelim-dg}

Let $\mathcal{M}\subset\mathbb{R}^{3}$ be a regular surface, locally
parametrised by a smooth immersion
$g\colon U\to\mathcal{M}$, $(u,v)\mapsto g(u,v)$, with
$U\subset\mathbb{R}^{2}$ open.
The tangent plane $T_{p}\mathcal{M}$ at $p=g(u,v)$ is spanned by
$g_{u}=\partial g/\partial u$ and $g_{v}=\partial g/\partial v$.
The \emph{Gauss map} $N\colon\mathcal{M}\to S^{2}$ sends each $p\in\mathcal{M}$
to its unit normal $N(p)=(g_{u}\times g_{v})/|g_{u}\times g_{v}|$.

The first and second fundamental forms of $\mathcal{M}$ are given as
\begin{align*}
    &E=\langle g_{u},g_{u}\rangle,\;
    F=\langle g_{u},g_{v}\rangle,\;
    G=\langle g_{v},g_{v}\rangle;\\
    &L=\langle g_{uu},\mathbf{n}\rangle,\;
    M=\langle g_{uv},\mathbf{n}\rangle,\;
    N=\langle g_{vv},\mathbf{n}\rangle.
\end{align*}
We follow the convention common in the surface-theoretic literature in
using $N$ for both the Gauss map and the $vv$-coefficient of the
second fundamental form; the meaning is always clear from context.
The \emph{Gauss curvature} of $\mathcal{M}$ is
\[
    K = \frac{LN-M^{2}}{EG-F^{2}},
\]
and the two eigenvalues of the second fundamental form
 are the \emph{principal curvatures} $k_{1}, k_{2}$,
with $K=k_{1}k_{2}$.

A point $p\in\mathcal{M}$ is called \emph{elliptic}, \emph{hyperbolic},
or \emph{parabolic} according to whether $K(p)>0$, $K(p)<0$, or
$K(p)=0$.
A point with $k_{1}(p)=k_{2}(p)$ is \emph{umbilic};
in particular, $k_{1}(p)=k_{2}(p)=0$ defines a \emph{flat umbilic}.
The \emph{parabolic set} of $\mathcal{M}$ is
\[
    \Sigma \;:=\; \{p\in\mathcal{M}\;|\;K(p)=0\}.
\]
At a non-flat parabolic point $p\in\Sigma$, the second fundamental
form has rank $1$, and the unique direction
$\mathbf{w}\in T_{p}\mathcal{M}\setminus\{0\}$
on which the second fundamental form vanishes is the
\emph{asymptotic direction} at $p$.
The set $\Sigma$ coincides with the singular set of the Gauss map
$N$, in the sense that the differential $dN_{p}$ fails to have full
rank precisely at $p\in\Sigma$.
We refer to the image
\[
    \Gamma \;:=\; N(\Sigma)\subset S^{2}
\]
as the \emph{spherical image} of $\Sigma$.

\subsection{Curves on $S^{2}$ and geodesic curvature}
\label{subsec:prelim-curves}

For a smooth curve $g\colon I\to S^{2}$ defined on an interval
$I\subset\mathbb{R}$, the \emph{geodesic curvature} of $g$ at $t\in I$
is
\begin{equation}\label{eq:prelim-geodesic-curvature}
    \kappa(t) \;=\;
    \frac{\langle g''(t),\;g'(t)\times g(t)\rangle}{|g'(t)|^{\,3}},
\end{equation}
where the cross product is taken in $\mathbb{R}^{3}$.
The vector $g'(t)\times g(t)/|g'(t)|$ is the geodesic unit normal of
$g$, lying in the tangent plane to $S^{2}$ at $g(t)$ and perpendicular
to $g'(t)$.
The curve $g$ is a \emph{geodesic} of $S^{2}$ (i.e.\ lies on a great
circle) if and only if $\kappa\equiv 0$.

A point $g(t_{0})$ on a non-geodesic curve $g$ is called an
\emph{inflection of geodesic curvature} (or, briefly, an
\emph{inflection point}) if $\kappa(t_{0})=0$;
it is a \emph{vertex} (of geodesic curvature) if
$\kappa(t_{0})\neq 0$ and $\kappa'(t_{0})=0$.
These notions, applied to the spherical curve $N\circ\gamma$ along
the parabolic set $\Sigma$, motivate the definitions of
\emph{inflection of Gauss} and \emph{vertex of Gauss} given in
Section~\ref{sec:intro}.

\subsection{Singularity types of two-variable functions}
\label{subsec:prelim-sing}

Two function germs
$f_{1}, f_{2}\colon(\mathbb{R}^{2},0)\to(\mathbb{R},0)$ are
\emph{$\mathcal{K}$-equivalent} ($f_{1}\sim_{\mathcal{K}}f_{2}$)
if there exist a diffeomorphism germ
$\varphi\colon(\mathbb{R}^{2},0)\to(\mathbb{R}^{2},0)$ and a
non-vanishing smooth function germ
$M\colon(\mathbb{R}^{2},0)\to\mathbb{R}^{*}$
such that
\[
    M\,f_{1} \;=\; f_{2}\circ\varphi.
\]

Among the simple $\mathcal{K}$-singularities of two-variable functions
classified by Arnol'd  \cite{AGV1985}, the $A_{k}$-series is given,
for $k\ge 1$, by the normal forms
\begin{equation}\label{eq:prelim-Ak-normal-form}
    A_{k}^{\pm}\colon\quad u^{2} \;\pm\; v^{k+1}.
\end{equation}
A function germ $f$ has an \emph{$A_{k}$-singularity}
(or is of type $A_{k}$) at the origin if
$f\sim_{\mathcal{K}}u^{2}+\epsilon\,v^{k+1}$
for some $\epsilon\in\{+1,-1\}$;
the choice of $\epsilon$ refines the type into $A_{k}^{\pm}$ when
relevant.
For $k=1$, the form $u^{2}+v^{2}$ corresponds to a definite Hessian
(Morse extremum), denoted $A_{1}^{+}$, while $u^{2}-v^{2}$ corresponds
to an indefinite Hessian (Morse saddle), denoted $A_{1}^{-}$.
A function germ is called \emph{Morse} if its Hessian at the origin
is non-degenerate, equivalently if it has an $A_{1}$-singularity, and
\emph{non-Morse} otherwise.

We use the notation $A_{\ge 2}$ for any non-Morse singularity.
A function germ has an $A_{\ge 2}$-singularity at the origin if and
only if its Hessian at the origin has rank at most $1$.

\section{Algebraic criteria for the slant function}
\label{sec:criteria}

We derive algebraic conditions for
$I_{\mathbf{v}}$ to have a critical point at $p$ and for the
singularity type to be $A_{\ge 2}$ or $A_k$.

\subsection{Monge form and the slant function}
\label{subsec:setup}

Let $p\in\mathcal{M}$.
The surface $\mathcal{M}$ around $p$ is locally written in \emph{Monge form}
\begin{equation}\label{eq:Monge}
    z \;=\; Q(u,v)
    \;=\; \frac{1}{2}\bigl(k_{1}\,u^{2} + k_{2}\,v^{2}\bigr)
        \;+\; \sum_{k\ge 3}H_{k}(u,v),
    \qquad
    H_{k}(u,v) := \!\sum_{i+j=k}\!\frac{a_{ij}}{i!\,j!}\,u^{i}\,v^{j},
\end{equation}
where $k_{1}$ and $k_{2}$ are the principal curvatures of $\mathcal{M}$
at $p$.
Equivalently, $a_{20}=k_{1}$, $a_{11}=0$, $a_{02}=k_{2}$.
The unit normal vector field is
\begin{equation}\label{eq:N}
    N(u,v) \;=\; \frac{1}{\sqrt{1+Q_{u}^{\,2}+Q_{v}^{\,2}}}
        \bigl(-Q_{u},\,-Q_{v},\,1\bigr),
\end{equation}
so $N(0)=(0,0,1)$.

We parametrise the light direction $\mathbf{v}\in S^{2}$ by spherical
coordinates
\begin{equation}\label{eq:spherical}
    \mathbf{v}_{\theta,\varphi}
    \;=\;\bigl(\sin\theta\cos\varphi,\;\sin\theta\sin\varphi,\;\cos\theta\bigr),
    \qquad\theta\in[0,\pi],\;\varphi\in[0,2\pi),
\end{equation}
so that $\theta=0$ corresponds to $\mathbf{v}=N(p)$, while $\theta=\pi/2$
parametrises the equator $T_{p}\mathcal{M}\cap S^{2}$.
The slant function takes the explicit form
\begin{equation}\label{eq:I-explicit}
    I_{\mathbf{v}}(u,v)
    \;=\;
    \frac{\cos\theta - Q_{u}\sin\theta\cos\varphi
                    - Q_{v}\sin\theta\sin\varphi}
         {\sqrt{1+Q_{u}^{\,2}+Q_{v}^{\,2}}}.
\end{equation}

We have the following rough stratification according to $(k_{1},k_{2})$:
$p$ is elliptic ($k_{1}k_{2}>0$), hyperbolic ($k_{1}k_{2}<0$), non-flat
parabolic ($k_{1}k_{2}=0$ but $(k_{1},k_{2})\neq(0,0)$), or a flat
umbilic ($k_{1}=k_{2}=0$).
At a non-flat parabolic point, we may further assume
\begin{equation}\label{eq:para-conv}
    k_{1}=0,\qquad k_{2}\neq 0,
\end{equation}
in which case the asymptotic direction at $p$ is the $u$-axis and the
asymptotic normal section at $p$ (introduced in
Section~\ref{sec:intro}) is the great circle
\(
    \{(\sin\theta,\,0,\,\cos\theta)\,:\,\theta\in[0,2\pi)\}\subset S^{2},
\)
i.e.\ the locus $\{\sin\varphi=0\}$ in the parametrisation
\eqref{eq:spherical}.

\subsection{First-jet analysis: critical directions}
\label{subsec:1jet}

A direct computation from~\eqref{eq:I-explicit} yields the linear part
of $I_{\mathbf{v}}$ at the origin.

\begin{lem}\label{lem:1jet}
The $1$-jet of $I_{\mathbf{v}}$ at the origin is
\[
    j^{1}I_{\mathbf{v}}(0)(u,v)
    \;=\;
    \cos\theta
    -\sin\theta\bigl(k_{1}\cos\varphi\,\cdot u
                  +k_{2}\sin\varphi\,\cdot v\bigr).
\]
In particular,
\[
    \nabla I_{\mathbf{v}}(0)
    \;=\;
    -\sin\theta\,\bigl(k_{1}\cos\varphi,\;k_{2}\sin\varphi\bigr).
\]
\end{lem}


We say $\mathbf{v}\in S^{2}$ is a \emph{critical direction} at $p$ if
$\nabla I_{\mathbf{v}}(0)=0$, i.e.\ $I_{\mathbf{v}}$ has a critical
point at $p$.
We write $\mathcal{C}_{p}$ for the set of all critical directions at $p$.

\begin{prop}\label{prop:critical-directions}
\begin{enumerate}
    \item If $p$ is elliptic or hyperbolic, then $\mathcal{C}_{p}=\{\pm N(p)\}$.
    \item If $p$ is non-flat parabolic, then $\mathcal{C}_{p}$ is the
          asymptotic normal section at $p$.
    \item If $p$ is a flat umbilic, then $\mathcal{C}_{p}=S^{2}$.
\end{enumerate}
\end{prop}

\begin{proof}
By Lemma~\ref{lem:1jet}, $\mathbf{v}_{\theta,\varphi}\in\mathcal{C}_{p}$
if and only if $\sin\theta=0$ or $(k_{1}\cos\varphi,\,k_{2}\sin\varphi)=(0,0)$.
In case~1, $k_{1}k_{2}\neq 0$ forces $\sin\theta=0$, giving $\mathbf{v}=\pm N(p)$.
In case~2, the condition reduces to $\sin\theta=0$ or $\sin\varphi=0$.
As a subset of $S^{2}$, this union is precisely the great circle
$\{\sin\varphi=0\}$, since the pole $\sin\theta=0$ is already contained
in it; this is the asymptotic normal section.
Case~3 is immediate.
\end{proof}

\begin{rem}\label{rem:critical-includes-normal}
The normal direction $N(p)$ corresponds to $\theta=0$ in
\eqref{eq:spherical}, hence $\sin\theta=0$, and is therefore always a
critical direction --- regardless of whether $p$ is parabolic.
At elliptic and hyperbolic points it is the only one;
at a non-flat parabolic point it is one of an entire great circle of
critical directions.
\end{rem}

\subsection{Second-jet analysis: the Hessian and $A_{\ge 2}$-directions}
\label{subsec:2jet}

We next compute the $2$-jet of $I_{\mathbf{v}}$ at the origin and use
the Hessian to determine the $A_{\ge 2}$-directions.

\begin{lem}\label{lem:2jet}
The $2$-jet of $I_{\mathbf{v}}$ at the origin is
\begin{equation}\label{eq:2jet}
    j^{2}I_{\mathbf{v}}(0)(u,v)
    \;=\;
    \cos\theta
    \;-\;\sin\theta\bigl(k_{1}\cos\varphi\,\cdot u
                       +k_{2}\sin\varphi\,\cdot v\bigr)
    \;+\;\frac{1}{2}\bigl(A\,u^{2}+2B\,uv+C\,v^{2}\bigr),
\end{equation}
where
\begin{align}
    A &= -k_{1}^{\,2}\cos\theta
         -\sin\theta\,(a_{30}\cos\varphi + a_{21}\sin\varphi),
         \label{eq:A}\\
    B &= -\sin\theta\,(a_{21}\cos\varphi + a_{12}\sin\varphi),
         \label{eq:B}\\
    C &= -k_{2}^{\,2}\cos\theta
         -\sin\theta\,(a_{12}\cos\varphi + a_{03}\sin\varphi).
         \label{eq:C}
\end{align}
In particular,
\(
\mathrm{Hess}\,I_{\mathbf{v}}(0)
=\Bigl(\begin{smallmatrix} A & B \\ B & C \end{smallmatrix}\Bigr).
\)
\end{lem}


\begin{lem}\label{lem:detH}
\begin{align}
    \det\mathrm{Hess}\,I_{\mathbf{v}}(0)
    \;=\;& k_{1}^{\,2}k_{2}^{\,2}\,\cos^{2}\theta\nonumber\\
        &\;+\sin\theta\cos\theta\,
              \Bigl[(a_{12}k_{1}^{\,2}+a_{30}k_{2}^{\,2})\cos\varphi
                   +(a_{03}k_{1}^{\,2}+a_{21}k_{2}^{\,2})\sin\varphi\Bigr]
              \nonumber\\
        &\;+\sin^{2}\theta\,\bigl[
              (a_{30}a_{12}-a_{21}^{\,2})\cos^{2}\varphi
              +(a_{30}a_{03}-a_{21}a_{12})\cos\varphi\sin\varphi
              \nonumber\\
        &\hphantom{\;+\sin^{2}\theta\,\bigl[}
              +(a_{21}a_{03}-a_{12}^{\,2})\sin^{2}\varphi\bigr].
              \label{eq:detH}
\end{align}
\end{lem}


We now combine Proposition~\ref{prop:critical-directions} with
Lemma~\ref{lem:detH}.
A critical direction is an $A_{\ge 2}$-direction precisely when the
Hessian determinant in~\eqref{eq:detH} vanishes.

\begin{prop}[Stratification of $2$-jets]
\label{prop:2jet-stratification}
For $p\in\mathcal{M}$, the $A_{\ge 2}$-directions of $I_{\mathbf{v}}$
at $p$ are described as follows.
\begin{enumerate}
    \item \textbf{Elliptic or hyperbolic point} ($k_{1}k_{2}\neq 0$).
    At the unique critical direction $\mathbf{v}=N(p)$,
    $\mathrm{Hess}\,I_{\mathbf{v}}(0)
        =\mathrm{diag}(-k_{1}^{\,2},-k_{2}^{\,2})$
    is negative definite;
    hence $I_{\mathbf{v}}$ has an $A_{1}^{+}$-singularity at $p$, and
    no $A_{\ge 2}$-direction exists.

    \item \textbf{Non-flat parabolic point}, under
    the setting~\eqref{eq:para-conv} ($k_{1}=0$, $k_{2}\neq 0$).
    The set of critical directions is the asymptotic normal section
    $\{\sin\varphi=0\}$.
    Specialising~\eqref{eq:detH} to $\varphi=0$ and $k_{1}=0$,
    \begin{equation}\label{eq:detH-para}
        \det\mathrm{Hess}\,I_{\mathbf{v}}(0)
        \;=\;
        \sin\theta\,\Bigl[\,a_{30}\,k_{2}^{\,2}\cos\theta
                +(a_{30}a_{12}-a_{21}^{\,2})\sin\theta\,\Bigr],
    \end{equation}
    which vanishes if and only if $\sin\theta=0$ or
    \begin{equation}\label{eq:cosRule}
        a_{30}\,k_{2}^{\,2}\cos\theta
        +(a_{30}a_{12}-a_{21}^{\,2})\sin\theta \;=\; 0.
    \end{equation}
    The three exhaustive sub-cases are:
    \begin{enumerate}
        \item \emph{Fold of Gauss} ($a_{30}\neq 0$).
              Equation~\eqref{eq:cosRule} has a unique solution
              $\theta^{*}\in(0,\pi)$, distinct from $\theta=0$.
              The $A_{\ge 2}$-directions on the asymptotic normal
              section are exactly $\mathbf{v}=N(p)$ and
              $\mathbf{v}_{\theta^{*}\!,0}$;
              all other directions on the great circle are
              $A_{1}$-directions.
        \item \emph{Degenerate singularity of Gauss} ($a_{30}=0$).
              Equation~\eqref{eq:cosRule} reduces to
              $-a_{21}^{\,2}\sin\theta=0$, and the analysis branches
              on $a_{21}$:
        \begin{enumerate}
            \item[(b-1)] \emph{$\Sigma$-regular type} ($a_{21}\neq 0$).
                  The unique $A_{\ge 2}$-direction on the asymptotic
                  normal section is $\mathbf{v}=N(p)$;
                  all other directions on the great circle are
                  $A_{1}$-directions.
            \item[(b-2)] \emph{$\Sigma$-singular type} ($a_{21}=0$).
                  $\det\mathrm{Hess}\,I_{\mathbf{v}}(0)\equiv 0$ on the
                  asymptotic normal section, so every direction on the
                  great circle is an $A_{\ge 2}$-direction and no
                  $A_{1}$-direction occurs among the critical directions.
        \end{enumerate}
    \end{enumerate}

    \item \textbf{Flat umbilic} ($k_{1}=k_{2}=0$).
    Every direction $\mathbf{v}\in S^{2}$ is a critical direction;
    in particular $I_{\mathbf{v}}$ has a singularity at $p$ for every
    $\mathbf{v}\in S^{2}$.
\end{enumerate}
\end{prop}

\begin{proof}
For case~1, the unique critical direction is $\mathbf{v}=N(p)$,
i.e.\ $\sin\theta=0$ and $\cos\theta=1$.
Setting $\sin\theta=0$ in~\eqref{eq:detH} gives
$k_{1}^{\,2}k_{2}^{\,2}>0$, and \eqref{eq:A}--\eqref{eq:C} at $\theta=0$
yield
$\mathrm{Hess}\,I_{\mathbf{v}}(0)=\mathrm{diag}(-k_{1}^{\,2},-k_{2}^{\,2})$,
which is negative definite under the elliptic and hyperbolic
hypotheses.
Hence $I_{\mathbf{v}}$ has a Morse singularity of type $A_{1}^{+}$
at $p$, and no $A_{\ge 2}$-direction exists.

For case~2, Proposition \ref{prop:critical-directions} allows us to restrict the light direction
to the great circle $\sin\varphi=0$.
Choosing the representative $\varphi=0$ (the case $\varphi=\pi$ is
antipodal, and produces an identical analysis up to sign),
and using~\eqref{eq:para-conv}, the
formula~\eqref{eq:detH} simplifies to~\eqref{eq:detH-para}.
\smallskip\noindent
\emph{Sub-case (2a):} If $a_{30}\neq 0$, then~\eqref{eq:cosRule} is a
nondegenerate linear equation in $\cos\theta$ and $\sin\theta$, with a
unique solution $\theta^{*}\in(0,\pi)$ characterised by
\[
    \tan\theta^{*}
    \;=\;
    \frac{a_{30}\,k_{2}^{\,2}}{a_{21}^{\,2}-a_{30}a_{12}}
\]
(when the denominator vanishes, $\theta^{*}=\pi/2$).
Combined with the solution $\theta=0$ of $\sin\theta=0$, we obtain
exactly two $A_{\ge 2}$-directions on the asymptotic normal section.
For any other $\theta\in(0,\pi)$, $\theta\neq\theta^{*}$, the bracketed
factor in~\eqref{eq:detH-para} is nonzero, and combined with
$\sin\theta\neq 0$ we have
$\det\mathrm{Hess}\,I_{\mathbf{v}}(0)\neq 0$, so $\mathbf{v}_{\theta,0}$
is an $A_{1}$-direction.

\smallskip\noindent
\emph{Sub-case (2b):} Assume $a_{30}=0$.
The bracketed factor in~\eqref{eq:detH-para} reduces to
$-a_{21}^{\,2}\sin\theta$, and
the analysis branches according to $a_{21}$.

\emph{(2b-1)} If, in addition, $a_{21}\neq 0$, then
$-a_{21}^{\,2}\sin\theta$ vanishes if and only if $\sin\theta=0$.
Hence the unique $A_{\ge 2}$-direction on the asymptotic normal section
is the normal direction $\mathbf{v}=N(p)$; every other direction
$\mathbf{v}_{\theta,0}$ with $\theta\in(0,\pi)$ is an $A_{1}$-direction.

\emph{(2b-2)} If, in addition, $a_{21}=0$, then~\eqref{eq:detH-para}
becomes $\det\mathrm{Hess}\,I_{\mathbf{v}}(0)=0$ {identically} on
the asymptotic normal section.
Every direction on the great circle satisfies the
$A_{\ge 2}$-condition, and no $A_{1}$-direction occurs among the
critical directions.

For case 3, $\nabla I_{\mathbf{v}}(0)\equiv 0$ for every $\mathbf{v}$ by
Proposition~\ref{prop:critical-directions}.
\end{proof}

\begin{rem}\label{rem:Sigma-regular}
At a non-flat parabolic point $p\in\mathcal{M}$ in
Monge form~\eqref{eq:Monge} under the setting~\eqref{eq:para-conv},
the parabolic set $\Sigma=\{LN-M^{2}=0\}$ has $1$-jet at the origin
$k_{2}(a_{30}u+a_{21}v)$. Hence $\Sigma$ is regular at $p$
if and only if $(a_{30},a_{21})\neq(0,0)$.
The three exhaustive sub-cases of
Proposition~\ref{prop:2jet-stratification}~(2) thus correspond to:
\begin{itemize}
\item$(a_{30}\neq 0)$ -- fold of Gauss;
\item$(a_{30}=0,\,a_{21}\neq 0)$ -- $\Sigma$ regular but $N|_{\Sigma}$
singular at $p$;
\item$(a_{30}=a_{21}=0)$ -- $\Sigma$ singular at $p$.
\end{itemize}
The first two are the $\Sigma$-regular settings traditionally treated in the literature \cite{BGM1982,IRFT}, while the third is a
higher-codimension stratum \cite{HKS2024, Honda2021}.
\end{rem}

\begin{rem}\label{rem:thm11-2jet-part}
Proposition~\ref{prop:2jet-stratification} establishes
Theorem~\ref{thm:stratification} at the level of the $2$-jet of
$I_{\mathbf{v}}$.
The correspondence between cases is as follows:
\begin{itemize}
    \item Case 1 of Theorem~\ref{thm:stratification}
          (elliptic or hyperbolic point) is completely settled by
          Proposition~\ref{prop:2jet-stratification}~(1),
          with the type sharpened from $A_{1}$ to $A_{1}^{+}$.
    \item Case 2 (fold of Gauss) corresponds to
          Proposition~\ref{prop:2jet-stratification}~(2a):
          the location and number of $A_{\ge 2}$-directions are
          settled, and only the precise type $A_{3}$ at the normal
          direction $\mathbf{v}=N(p)$ remains for the higher jets.
    \item Case 3 (degenerate singularity of Gauss with $\Sigma$
          regular at $p$) corresponds to
          Proposition~\ref{prop:2jet-stratification}~(2b-1): again
          the location of $A_{\ge 2}$-directions is settled, and only
          the precise type $A_{\ge 5}$ at $\mathbf{v}=N(p)$ remains.
    \item Case 4 (degenerate singularity of Gauss with $\Sigma$
          singular at $p$) corresponds to
          Proposition~\ref{prop:2jet-stratification}~(2b-2) and is
          completely settled at the $2$-jet level: the entire
          asymptotic normal section consists of $A_{\ge 2}$-directions.
    \item Case 5 (flat umbilic point) is settled by
          Proposition~\ref{prop:2jet-stratification}~(3).
\end{itemize}
The remaining type identifications $A_{3}$ in case~2 and $A_{\ge 5}$
in case~3 are obtained from the $3$- and higher-order jets, which we
develop in the next subsection.
\end{rem}

\begin{rem}\label{rem:second-direction-equation}
Equation~\eqref{eq:cosRule} will play a central role in the higher-order
analysis: at a fold of Gauss, it is the unique condition imposed by
$A_{\ge 2}$-degeneracy at the second $A_{\ge 2}$-direction
$\mathbf{v}_{\theta^{*}\!,0}$, and substituting it back into the higher
jets reduces the singularity-type question for
$\mathbf{v}_{\theta^{*}\!,0}$ to a polynomial condition in
$a_{ij}\,(i+j\le 4)$ alone.
This is carried out in \S\ref{subsec:higher-jets}.
\end{rem}

\subsection{Higher-jet analysis}
\label{subsec:higher-jets}

The main tool is Fukui's criterion \cite{Fukui2011} for identifying
the precise singularity type at each $A_{\ge 2}$-direction; we recall
it in \S\ref{subsubsec:fukui}.

\subsubsection{Fukui's $A_{k}$-criterion}
\label{subsubsec:fukui}

Let $f\colon(\mathbb{R}^{2},0)\to(\mathbb{R},0)$ be a smooth function
germ with $df(0)=0$, and write
\begin{equation}\label{eq:f-expansion}
    f(u,v)\;=\;\sum_{k\ge 2}f_{k}(u,v),
    \qquad
    f_{k}(u,v) \;=\;
    \!\sum_{i+j=k}\!\frac{c_{ij}}{i!\,j!}\,u^{i}\,v^{j},
\end{equation}
so that $f_{k}$ is the homogeneous part of degree $k$.
For $i+j\le k$, set
$f_{k,ij}:=\partial^{i+j}f_{k}/\partial u^{i}\,\partial v^{j}$,
a homogeneous polynomial of degree $k-(i+j)$, and let $H_{f}(0)$
denote the Hessian of $f$ at the origin.

\begin{lem}[Rank-one form, Fukui {\cite[\S 1.1]{Fukui2011}}]
\label{lem:rank-one}
The Hessian $H_{f}(0)$ has rank $1$ if and only if there exist
$(du,dv)\neq(0,0)$ and $s\neq 0$ with
\begin{equation}\label{eq:rank-one}
    H_{f}(0) \;=\; s\,
    \begin{pmatrix} dv^{\,2} & -du\,dv\\
                    -du\,dv & du^{\,2} \end{pmatrix}.
\end{equation}
The pair $(du,dv)$ is unique up to a nonzero scalar and gives
the direction of $\ker H_{f}(0)$, and $s$ is then determined
by~\eqref{eq:rank-one}.
\end{lem}

In what follows, $(du,dv)$ and $s$ refer to those satisfying
\eqref{eq:rank-one}.

\begin{thm}[$A_{k}$-criterion, Fukui {\cite[Theorems 1.1--1.4]{Fukui2011}}]
\label{thm:fukui}
Suppose $H_{f}(0)$ has rank one with $(du,dv)$ and $s$ as
in~\eqref{eq:rank-one}.
\begin{enumerate}
    \item $f$ has an $A_{2}$-singularity at $0$ if and only if
    $\;f_{3}(du,dv)\neq 0$.
    \item Suppose $f_{3}(du,dv)=0$ and $du\neq 0$.
    Then $f$ has an $A_{3}$-singularity at $0$ if and only if
    \begin{equation}\label{eq:fukui-A3}
        f_{4}(du,dv)\;-\;
        \frac{1}{2\,s\,du^{\,2}}\,f_{3,01}(du,dv)^{2}
        \;\neq\; 0.
    \end{equation}
    \item Suppose~\eqref{eq:fukui-A3} vanishes and $du\neq 0$.
    Then $f$ has an $A_{4}$-singularity at $0$ if and only if
    \begin{equation}\label{eq:fukui-A4}
    \begin{aligned}
        & f_{5}(du,dv)
        \;-\;\frac{1}{s\,du^{\,2}}\,
            f_{3,01}(du,dv)\,f_{4,01}(du,dv)\\
        &\qquad+\;\frac{1}{2\,s^{2}\,du^{\,4}}\,
            f_{3,02}(du,dv)\,f_{3,01}(du,dv)^{2}
        \;\neq\;0.
    \end{aligned}
    \end{equation}
    \item Suppose~\eqref{eq:fukui-A4} vanishes and $du\neq 0$.
    Then $f$ has an $A_{5}$-singularity at $0$ if and only if
    \begin{equation}\label{eq:fukui-A5}
    \begin{aligned}
        & f_{6}(du,dv)\\
        &-\frac{1}{2\,s\,du^{\,2}}\,
          \Bigl[2\,f_{3,01}(du,dv)\,f_{5,01}(du,dv)
                +f_{4,01}(du,dv)^{2}\Bigr]\\
        &+\frac{1}{2\,s^{2}\,du^{\,4}}\,
          f_{3,01}(du,dv)\,
          \Bigl[f_{3,01}(du,dv)\,f_{4,02}(du,dv)\\
        &\hphantom{+\frac{1}{2\,s^{2}\,du^{\,4}}\,
          f_{3,01}(du,dv)\,\Bigl[}
          +2\,f_{3,02}(du,dv)\,f_{4,01}(du,dv)\Bigr]\\
        &-\frac{1}{6\,s^{3}\,du^{\,6}}\,
          f_{3,01}(du,dv)^{2}\,
          \Bigl[f_{3,03}(du,dv)\,f_{3,01}(du,dv)\\
        &\hphantom{-\frac{1}{6\,s^{3}\,du^{\,6}}\,
          f_{3,01}(du,dv)^{2}\,\Bigl[}
          +3\,f_{3,02}(du,dv)^{2}\Bigr]
        \;\neq\;0.
    \end{aligned}
    \end{equation}
\end{enumerate}
If $du=0$, the formulas in the statements 2--4 hold with $f_{k,ij}$ replaced
by $f_{k,ji}$ and $1/du$ replaced by $1/dv$.
\end{thm}

\subsubsection{Type at the normal direction}
\label{subsubsec:type-normal}

\begin{prop}[Type at the normal direction]
\label{prop:type-normal}
At a non-flat parabolic point $p$ under~\eqref{eq:para-conv}, the
slant function $I_{N(p)}$ along the normal direction has the following
singularity type.
\begin{enumerate}
    \item If $a_{30}\neq 0$, $I_{N(p)}$ has
          an $A_{3}$-singularity.
    \item If $a_{30}=0$ and $a_{21}\neq 0$, $I_{N(p)}$ has
          an $A_{\ge 5}$-singularity.
          Moreover, if in addition $a_{02}\,a_{40}-3\,a_{21}^{\,2}\neq 0$,
          $I_{N(p)}$ has an $A_{5}$-singularity.
\end{enumerate}
\end{prop}

\begin{proof}
We apply Theorem~\ref{thm:fukui} to $I_{N(p)}$
under the setting~\eqref{eq:para-conv}.
By~\eqref{eq:A}--\eqref{eq:C} with $\theta=0$,
$\mathrm{Hess}\,I_{N(p)}(0)=\mathrm{diag}(0,-a_{02}^{\,2})$,
which is rank one with $(du,dv)=(1,0)$ and $s=-a_{02}^{\,2}$.
A direct calculation from~\eqref{eq:I-explicit} gives
\begin{equation}\label{eq:cij-normal}
\begin{aligned}
    c_{30}=0,\;\;
     c_{21}=-a_{02}\,a_{21},\;\;
     c_{12}=-2\,a_{02}\,a_{12},\;\;
     c_{03}=-3\,a_{02}\,a_{03},\\
    c_{40}=-3\,(a_{30}^{\,2}+a_{21}^{\,2}),\;\;
    c_{31}=-a_{02}\,a_{31}-3\,a_{21}\,a_{30}-3\,a_{12}\,a_{21},
\end{aligned}
\end{equation}
together with formulas for $c_{22}, c_{13}, c_{04}$ which we shall not need.
Since $c_{30}=0$, we have $f_{3}(1,0)=0$ and $I_{N(p)}$ is at least $A_{3}$.
With $f_{3,01}=\partial f_{3}/\partial v$,
\[
    f_{4}(1,0)=\frac{c_{40}}{24}=-\frac{a_{30}^{\,2}+a_{21}^{\,2}}{8},
    \qquad
    f_{3,01}(1,0)=\frac{c_{21}}{2}=-\frac{a_{02}\,a_{21}}{2},
\]
so the $A_{3}$ criterion~\eqref{eq:fukui-A3} evaluates to
\begin{equation}\label{eq:A3-normal}
    f_{4}(1,0)\;-\;\frac{1}{2\,s}\,f_{3,01}(1,0)^{2}
    \;=\;
    -\frac{a_{30}^{\,2}+a_{21}^{\,2}}{8}\;-\;
    \frac{1}{-2\,a_{02}^{\,2}}\cdot\frac{a_{02}^{\,2}\,a_{21}^{\,2}}{4}
    \;=\;
    -\frac{a_{30}^{\,2}}{8}.
\end{equation}
Hence $I_{N(p)}$ has an $A_{3}$-singularity if and only if $a_{30}\neq 0$,
which proves case~1.

For case~2, with $a_{30}=0$ the expression~\eqref{eq:A3-normal} vanishes,
so $I_{N(p)}$ is at least $A_{4}$.
Extending the Taylor expansion~\eqref{eq:cij-normal} to higher order
and applying the criteria~\eqref{eq:fukui-A4}--\eqref{eq:fukui-A5}
by the same method, one finds that the $A_{4}$ criterion vanishes
identically (so $I_{N(p)}$ is at least $A_{5}$), and the $A_{5}$
criterion evaluates to
\(
    -(a_{02}\,a_{40}-3\,a_{21}^{\,2})^{2}/(72\,a_{02}^{\,2}).
\)
\end{proof}

\begin{rem}\label{rem:thm11-completed}
Combining Proposition~\ref{prop:type-normal} with
Proposition~\ref{prop:2jet-stratification}, the type identifications
required for cases~(2) and~(3) of Theorem~\ref{thm:stratification}
are now complete: at a fold of Gauss the normal direction has type
$A_{3}$, and at a $\Sigma$-regular degenerate singularity of Gauss
with $a_{02}\,a_{40}-3\,a_{21}^{\,2}\neq 0$ the normal direction has
type $A_{5}$.
Together with the $2$-jet analysis of \S\ref{subsec:2jet}, this
completes the proof of Theorem~\ref{thm:stratification}.
\end{rem}

\subsubsection{Type at the second $A_{\ge 2}$-direction}
\label{subsubsec:type-second}

We now treat the second $A_{\ge 2}$-direction at a fold of Gauss:
the direction $\mathbf{v}_{\theta^{*}\!,0}$, with $\theta^{*}\in(0,\pi)$
satisfying~\eqref{eq:cosRule}, under the assumption $a_{30}\neq 0$.
A direct calculation from Lemma~\ref{lem:2jet} and~\eqref{eq:cosRule}
gives
\begin{equation}\label{eq:Hess-second}
    \mathrm{Hess}\,I_{\mathbf{v}_{\theta^{*}\!,0}}(0)
    \;=\;
    -\,\frac{\sin\theta^{*}}{a_{30}}
    \begin{pmatrix} a_{30}\\ a_{21}\end{pmatrix}
    \begin{pmatrix} a_{30} & a_{21}\end{pmatrix},
\end{equation}
which is in rank-one form~\eqref{eq:rank-one} with
\begin{equation}\label{eq:second-rank-one}
    (du,dv) = (-a_{21},\,a_{30}),
    \qquad
    s \;=\; -\,\frac{\sin\theta^{*}}{a_{30}}\;\neq\;0.
\end{equation}

\paragraph{Notation.}
For the homogeneous polynomials $H_{k}(u,v)$ of degree $k$ appearing
in the Monge expansion~\eqref{eq:Monge}, we write
$H_{k,u}:=\partial H_{k}/\partial u$,
$H_{k,v}:=\partial H_{k}/\partial v$,
$H_{k,uu}:=\partial^{2}H_{k}/\partial u^{2}$,
$H_{k,vv}:=\partial^{2}H_{k}/\partial v^{2}$, etc., for partial
derivatives.
By the homogeneity of $H_{k}$, each such derivative is itself a
homogeneous polynomial in $(u,v)$ of degree $k$ minus the order of
differentiation, and we evaluate it at $(u,v)=(-a_{21},a_{30})$.

Define the homogeneous polynomials
\begin{align}
    \alpha_{2} &\;:=\;
    a_{02}\,{H_{4}}_{,u}(-a_{21},a_{30})
    \;-\;(a_{30}a_{12}-a_{21}^{\,2})\,{H_{3}}_{,v}(-a_{21},a_{30}),
    \label{eq:alpha2}
    \\[0.4em]
    \alpha_{3} &\;:=\;
    \begin{aligned}[t]
    & \Bigl[-8\,a_{30}\,{H_{5}}_{,u}(-a_{21},a_{30})
      +4\,{H_{4}}_{,uu}(-a_{21},a_{30})^{2}\\
    &\;-\,a_{02}^{\,2}\,a_{30}^{\,4}\,(a_{30}a_{12}-a_{21}^{\,2})\Bigr]\,a_{02}^{\,2}\\
    &-\,8\,a_{30}\,(a_{30}a_{12}-a_{21}^{\,2})\,
      \Bigl[a_{12}\,{H_{4}}_{,uu}(-a_{21},a_{30})
            +2\,{H_{4}}_{,v}(-a_{21},a_{30})\\
    &\hphantom{-\,8\,a_{30}\,(a_{30}a_{12}-a_{21}^{\,2})\,\Bigl[}
            -\,a_{30}\,{H_{4}}_{,vv}(-a_{21},a_{30})\Bigr]\,a_{02}\\
    &-\,(a_{30}a_{12}-a_{21}^{\,2})^{2}\,
      \Bigl[9\,(a_{30}a_{12}-a_{21}^{\,2})^{2}
            +a_{30}\,a_{12}\,(a_{30}a_{12}-a_{21}^{\,2})\\
    &\hphantom{-\,(a_{30}a_{12}-a_{21}^{\,2})^{2}\,\Bigl[}
            +14\,a_{30}^{\,2}\,(a_{21}\,a_{03}-a_{12}^{\,2})\Bigr].
    \end{aligned}
    \label{eq:alpha3}
\end{align}

\begin{prop}[Type at the second $A_{\ge 2}$-direction]
\label{prop:type-second}
At a fold of Gauss $p$ ($a_{30}\neq 0$), the slant function
$I_{\mathbf{v}_{\theta^{*}\!,0}}$ along the second $A_{\ge 2}$-direction
has the following singularity type.
\begin{enumerate}
    \item $I_{\mathbf{v}_{\theta^{*}\!,0}}$ has an
          $A_{2}$-singularity at $p$ if and only if $\alpha_{2}\neq 0$.
    \item Suppose $\alpha_{2}=0$.
          Then $I_{\mathbf{v}_{\theta^{*}\!,0}}$ has an
          $A_{3}$-singularity at $p$ if and only if $\alpha_{3}\neq 0$.
\end{enumerate}
\end{prop}

\begin{proof}
We apply Theorem~\ref{thm:fukui} with~\eqref{eq:second-rank-one}.
A direct calculation from~\eqref{eq:I-explicit} and~\eqref{eq:cosRule}
gives
\begin{equation}\label{eq:f3-second}
    f_{3}(-a_{21},a_{30})
    \;=\;
    -\,\frac{\sin\theta^{*}}{a_{02}}\,\alpha_{2}.
\end{equation}
Theorem~\ref{thm:fukui}~(1) then gives statement~(1).

For statement~(2), assume $\alpha_{2}=0$.
We solve $\alpha_{2}=0$ for $a_{13}$:
\begin{equation}\label{eq:a13-solved}
\begin{split}
    a_{13} \;=\; -\frac{1}{a_{02}\,a_{30}^{\,3}}\,\Bigl\{&
        a_{02}\,a_{21}\bigl(2\,{H_{4}}_{,uu}(-a_{21},a_{30})
                        -a_{30}(a_{21}\,a_{31}-2\,a_{30}\,a_{22})\bigr)\\
        &+\,6\,(a_{30}a_{12}-a_{21}^{\,2})\,
              {H_{3}}_{,v}(-a_{21},a_{30})\Bigr\}.
\end{split}
\end{equation}
Applying the $A_{3}$-criterion~\eqref{eq:fukui-A3}
with~\eqref{eq:a13-solved} and~\eqref{eq:cosRule},
a direct calculation gives
\begin{equation}\label{eq:f4-second}
    f_{4}(-a_{21},a_{30})\;-\;
    \frac{1}{2\,s\,a_{21}^{\,2}}\,f_{3,01}(-a_{21},a_{30})^{2}
    \;=\;
    -\,\frac{\sin\theta^{*}}{3\,a_{02}^{\,2}\,a_{30}}\,\alpha_{3}.
\end{equation}
Theorem~\ref{thm:fukui}~(2) then gives statement~(2).
\end{proof}

\begin{rem}\label{rem:alpha-collapse}
The polynomials $\alpha_{2}$ and $\alpha_{3}$
in~\eqref{eq:alpha2}--\eqref{eq:alpha3} are defined uniformly on the
fold-of-Gauss stratum, with no case distinction.
Their explicit form nevertheless depends sensitively on the vanishing of
$a_{30}a_{12}-a_{21}^{\,2}$:
\begin{itemize}
    \item When $a_{30}a_{12}-a_{21}^{\,2}\neq 0$,
          $\theta^{*}$ takes any value in $(0,\pi)$ except $\pi/2$.
    \item When $a_{30}a_{12}-a_{21}^{\,2}=0$,
          equation~\eqref{eq:cosRule} reduces to
          $a_{30}\,a_{02}^{\,2}\cos\theta^{*}=0$, implying
          $\cos\theta^{*}=0$, i.e.\
          $\mathbf{v}_{\theta^{*}\!,0}=(1,0,0)$ is the asymptotic
          direction itself.
          In this special case, $\alpha_{2}$ collapses to
          $a_{02}\,{H_{4}}_{,u}(-a_{21},a_{30})$ and $\alpha_{3}$ to
          \(
            4\,a_{02}^{\,2}\bigl(-2\,a_{30}\,{H_{5}}_{,u}(-a_{21},a_{30})
            +{H_{4}}_{,uu}(-a_{21},a_{30})^{2}\bigr).
          \)
\end{itemize}
The geometric content of $\alpha_{2}$ and $\alpha_{3}$, and of
$a_{30}a_{12}-a_{21}^{\,2}$, can be expressed in terms of the geodesic curvature
$\kappa$ of the spherical curve $N\circ\gamma$ along $\Sigma$ at $p$:
the three quantities correspond, respectively, to (a nonzero scalar
multiple of) $\kappa(0)$, $\kappa'(0)$, and $\kappa''(0)$.
We develop this connection in~\S\ref{sec:proof}.
\end{rem}

\section{Proofs of the main theorems}
\label{sec:proof}

Proofs of both theorems rest on the geometric characterisations
of \S\ref{subsec:geometric-lemma}.

\subsection{Geodesic curvature of $N\circ\gamma$ in Monge coordinates}
\label{subsec:geometric-lemma}

Throughout this subsection we work in the Monge form~\eqref{eq:Monge}
under the setting~\eqref{eq:para-conv}: $k_{1}=0$ and
$k_{2}=a_{02}\neq 0$.
We recall from Remark~\ref{rem:Sigma-regular} that the parabolic set
$\Sigma=\{LN-M^{2}=0\}$ has $1$-jet $a_{02}\,(a_{30}u + a_{21}v)$ at
the origin, and is regular at $p$ if and only if
$(a_{30},a_{21})\neq(0,0)$.

\begin{lem}\label{lem:gamma-jet}
At a fold of Gauss $p$ ($a_{30}\neq 0$), $\Sigma$ admits a smooth
parametrisation $\gamma\colon I\to\mathbb{R}^{2}$, $\gamma(0)=0$, of the form
\begin{equation}\label{eq:gamma-form}
    \gamma(t) \;=\; \bigl(u(t),\; a_{30}\,t\bigr),
\end{equation}
whose $3$-jet is
\begin{equation}\label{eq:gamma-3jet}
    j^{3}\gamma(t) \;=\;
    \biggl(\,\sum_{i=1}^{3}\frac{c_{i}}{i!}\,t^{i},\; a_{30}\,t\,\biggr),
\end{equation}
where
\begin{align}
    c_{1} \;=\;& -a_{21},
    \label{eq:c1}\\
    c_{2} \;=\;&
    \frac{2}{a_{30}\,a_{02}}\,
    \Bigl[\,
        -a_{02}\,H_{4,uu}(-a_{21},a_{30})
        +(a_{30}a_{12}-a_{21}^{\,2})^{2}
    \,\Bigr],
    \label{eq:c2}\\
    c_{3} \;=\;& -\frac{6}{a_{30}^{\,2}\,a_{02}^{\,2}}\,\Bigl\{\,
    \bigl[
        a_{30}^{\,2}\,H_{5,uu}(-a_{21},a_{30})
        -H_{4,uuu}(-a_{21},a_{30})\,H_{4,uu}(-a_{21},a_{30})
    \bigr]\,a_{02}^{\,2}
    \nonumber\\
    & + a_{30}\,(a_{30}a_{12}-a_{21}^{\,2})\,
        \bigl[(a_{30}a_{12}-a_{21}^{\,2})\,H_{4,uuu}(-a_{21},a_{30})
            -6\,H_{4,u}(-a_{21},a_{30})\bigr]\,a_{02}
    \nonumber\\
    & + 6\,(a_{30}a_{12}-a_{21}^{\,2})^{2}\,H_{3}(-a_{21},a_{30})
    \,\Bigr\}.
    \label{eq:c3}
\end{align}
\end{lem}

\begin{proof}
The parabolic set $\Sigma$ is defined by $LN-M^{2}=Q_{uu}Q_{vv}-Q_{uv}^{\,2}=0$.
By Remark~\ref{rem:Sigma-regular}, the $1$-jet of $LN-M^{2}$ at the origin is
$a_{02}\,(a_{30}u+a_{21}v)$.
Since $a_{30}\neq 0$, the implicit function theorem yields a smooth parametrisation
of $\Sigma$ near the origin with tangent direction $(-a_{21},a_{30})$.
Taking $v=a_{30}\,t$ as the parameter yields the form~\eqref{eq:gamma-form}.
The coefficients $c_{i}$ are determined by substituting
$j^{3}\gamma(t)$ into $LN-M^{2}=0$ and comparing coefficients of $t^{i}$:
the $t^{1}$-term gives $c_{1}=-a_{21}$, the $t^{2}$-term gives~\eqref{eq:c2},
and the $t^{3}$-term gives~\eqref{eq:c3}.
\end{proof}

\begin{rem}\label{rem:3jet-suffices}
For the computation of $\kappa(0)$, $\kappa'(0)$, and $\kappa''(0)$ in
Theorem~\ref{thm:kappa-formulas} below, the $3$-jet
in~\eqref{eq:gamma-3jet} is sufficient; no higher-order information on
$\gamma$ is needed.
Geometrically, this reflects the fact that the asymptotic direction
$\partial_{u}$ at $p$ lies in $\ker DN(0)$; see the proof of
Theorem~\ref{thm:kappa-formulas}.
\end{rem}

We now compute the geodesic curvature of $N\circ\gamma$.
For a smooth curve $g\colon I\to S^{2}$, the (signed) geodesic
curvature is
\begin{equation}\label{eq:kappa-formula}
    \kappa \;=\;
    \frac{\langle g'',\;g'\times g\rangle}{|g'|^{\,3}},
\end{equation}
where $g\in S^{2}$ is identified with the outward unit normal to $S^{2}$
at $g$, so that $g'\times g/|g'|$ is the geodesic unit normal of the
curve.

\begin{thm}[Geodesic curvature of $N\circ\gamma$]
\label{thm:kappa-formulas}
At a fold of Gauss $p$ under~\eqref{eq:para-conv}, the geodesic
curvature $\kappa$ of $N\circ\gamma$ at $p$ satisfies
\begin{equation}\label{eq:kappa0}
    \kappa(0) \;=\;
    \frac{a_{30}\,a_{12}-a_{21}^{\,2}}{a_{02}^{\,2}\,a_{30}}.
\end{equation}
Moreover:
\begin{enumerate}
    \item If $a_{30}a_{12}-a_{21}^{\,2}\neq 0$, then
    \begin{equation}\label{eq:kappa-prime}
        \kappa'(0) \;=\;
        \frac{6\,\alpha_{2}}{a_{02}^{\,3}\,a_{30}^{\,2}},
    \end{equation}
    where $\alpha_{2}$ is as in~\eqref{eq:alpha2}.
    \item If $a_{30}a_{12}-a_{21}^{\,2}\neq 0$ and $\alpha_{2}=0$, then
    \begin{equation}\label{eq:kappa-pprime}
        \kappa''(0) \;=\;
        \frac{\alpha_{3}}{a_{02}^{\,4}\,a_{30}^{\,3}},
    \end{equation}
    where $\alpha_{3}$ is as in~\eqref{eq:alpha3}.
    \item If $a_{30}a_{12}-a_{21}^{\,2}=0$, then
    \begin{equation}\label{eq:kappa-prime-collapsed}
        \kappa'(0) \;=\;
        -\,\frac{H_{4,u}(-a_{21},a_{30})}{a_{02}^{\,2}\,a_{30}^{\,2}}.
    \end{equation}
\end{enumerate}
\end{thm}

\begin{proof}
Let $g=N\circ\gamma$ with $\gamma$ as in Lemma~\ref{lem:gamma-jet}.
Since $\kappa(t)$ is determined by $g(t)$, $g'(t)$, and $g''(t)$, the
values $\kappa(0)$, $\kappa'(0)$, and $\kappa''(0)$ are determined by
$j^{4}g$.

Under the Monge form~\eqref{eq:Monge} with $k_{1}=0$ and $a_{11}=0$,
a direct computation gives
\begin{equation}\label{eq:DN-kernel}
    N_{u}(0) \;=\; \bigl(-Q_{uu}(0),\,-Q_{uv}(0),\,0\bigr) \;=\; (0,0,0),
\end{equation}
so that the asymptotic direction $\partial_{u}$ at $p$ lies in
$\ker DN(0)$.
The Taylor expansion of $N$ at the origin therefore takes the form
\begin{equation}\label{eq:N-Taylor}
    N(u,v) \;=\; N(0) \;+\; N_{v}(0)\,v
    \;+\; \sum_{\substack{i+j\geq 2\\ i\geq 1}}
        \frac{\partial_{u}^{i}\partial_{v}^{j}N(0)}{i!\,j!}\,u^{i}v^{j}
    \;+\; \sum_{j\geq 2}\frac{\partial_{v}^{j}N(0)}{j!}\,v^{j},
\end{equation}
without a linear term in $u$.
Substituting $\gamma(t)=(u(t),a_{30}t)$ from~\eqref{eq:gamma-form},
each monomial $u(t)^{i}(a_{30}t)^{j}$ with $i\geq 1$ and $i+j\geq 2$
is of order at least $t^{i+j}$, so its contribution to the
$t^{4}$-coefficient of $g(t)$ involves $u(t)$ only up to order
$t^{\,5-i-j}\leq t^{3}$.
The pure $v$-terms do not involve $u(t)$.
Hence $j^{4}g$ is determined by $j^{3}\gamma$ together with the
partial derivatives $\partial_{u}^{i}\partial_{v}^{j}N(0)$ for
$i+j\leq 4$, equivalently the polynomials $H_{3}$, $H_{4}$, $H_{5}$ in
the Monge expansion of $Q$.

Substituting $j^{3}\gamma$ from~\eqref{eq:gamma-3jet} into $g=N\circ\gamma$
and expanding in powers of $t$ yields $j^{4}g$.
Applying~\eqref{eq:kappa-formula} and differentiating with respect to
$t$ gives~\eqref{eq:kappa0} and, when $a_{30}a_{12}-a_{21}^{\,2}\neq 0$,
also~\eqref{eq:kappa-prime}.

For~(2), the hypothesis $\alpha_{2}=0$ is solved for $a_{13}$
via~\eqref{eq:a13-solved}; substituting into $j^{4}g$ and differentiating
twice yields~\eqref{eq:kappa-pprime}.

For~(3), under $a_{12}=a_{21}^{\,2}/a_{30}$ the factor
$a_{30}a_{12}-a_{21}^{\,2}$ in~\eqref{eq:alpha2} vanishes, so
$\alpha_{2}$ reduces to $a_{02}H_{4,u}(-a_{21},a_{30})$, and statement~(1)
gives~\eqref{eq:kappa-prime-collapsed}.
\end{proof}

\begin{rem}
The results above concerning the inflection of Gauss appear,
explicitly or implicitly, in \cite{BGM1982, IRFT, KvD1993}.
\end{rem}

\begin{cor}
\label{cor:geom-alg}
Let $p$ be a parabolic point of $\mathcal{M}$ that is not a flat
umbilic, in Monge form under the setting~\eqref{eq:para-conv}.
The geometric strata defining the four cases of
Theorem~\ref{thm:parabolic} are equivalent to the following
algebraic conditions on the Monge coefficients:
\begin{enumerate}
    \item $p$ is a fold of Gauss but neither an inflection nor a
          vertex of Gauss
          $\;\Longleftrightarrow\;$
          $a_{30}\neq 0,\;\;a_{30}a_{12}-a_{21}^{\,2}\neq 0,\;\;
                 \alpha_{2}\neq 0.$
    \item $p$ is a first-order inflection of Gauss
          $\;\Longleftrightarrow\;$
          $a_{30}\neq 0,\;\;a_{30}a_{12}-a_{21}^{\,2}=0,\;\;
                 H_{4,u}(-a_{21},a_{30})\neq 0.$
    \item $p$ is a first-order vertex of Gauss
          $\;\Longleftrightarrow\;$
          $a_{30}\neq 0,\;\;a_{30}a_{12}-a_{21}^{\,2}\neq 0,\;\;
                 \alpha_{2}=0,\;\;\alpha_{3}\neq 0.$
    \item $p$ is a cusp of Gauss at which $\Sigma$ is regular
          $\;\Longleftrightarrow\;$
          $a_{30}=0,\;\;a_{21}\neq 0,\;\;
                 a_{02}\,a_{40}-3\,a_{21}^{\,2}\neq 0.$
\end{enumerate}
\end{cor}

\begin{proof}
For cases~1--3, the geometric conditions in
Theorem~\ref{thm:parabolic} are formulated as vanishing/non-vanishing
of $\kappa(0)$, $\kappa'(0)$, $\kappa''(0)$.
By Theorem~\ref{thm:kappa-formulas}, these translate directly into
the stated algebraic conditions:
case~(1) (neither inflection nor vertex) corresponds to
$\kappa(0)\neq 0$ and $\kappa'(0)\neq 0$;
case~(2) (first-order inflection) to $\kappa(0)=0$ and $\kappa'(0)\neq 0$;
case~(3) (first-order vertex) to $\kappa(0)\neq 0$, $\kappa'(0)=0$,
$\kappa''(0)\neq 0$.

For case~(4), by~\cite[Theorem~6.2]{IRFT}, the conditions $a_{30}=0$
and $a_{02}\,a_{40}-3\,a_{21}^{\,2}\neq 0$ characterise the
cusp of Gauss in Monge coordinates,
while $a_{21}\neq 0$ is equivalent to $\Sigma$ being regular at $p$
(Remark~\ref{rem:Sigma-regular}).
\end{proof}

\begin{rem}
\label{rem:literature-parabolic}
The stratification of parabolic points has been examined through several
frameworks in singularity theory---height functions, distance-squared
functions, the projection, the reflection map, the principal map, ridges,
sub-parabolic points, and binary differential equations, among others
\cite{BGM1982, BG1992, Bruce1984, BT2019, FH2012, FHN2017, FuHi2023, HKS2024,
Honda2021, IRFT, Kabata2016, Pla1986, PST2024, Por2001}.
Within each of these, the cusp of Gauss admits a well-known characterisation.
The inflection of Gauss has so far been identified only
with the goose singularity of the projection
\cite{BGM1982, Bruce1984, HKS2024, IRFT, Pla1986}, with a cusp
singularity of the principal map \cite{BT2019}, and with a bifurcation point at which the $A_1^{\pm}$-type classification
of the slant function undergoes a qualitative change \cite{Koen1990, KvD1993}.
(Koenderink and van Doorn call this point a \emph{gutterpoint},
following the terminology of \cite{KT1980}.)
At a fold of Gauss, the slant function admits a non-normal
$A_{\ge 2}$-direction in addition to $N(p)$;
Theorem~\ref{thm:parabolic} refines this characterisation by showing
that the inflection of Gauss is precisely the locus where this
non-normal $A_{\ge 2}$-direction coincides with the asymptotic direction.
The vertex of Gauss has not been previously identified in the literature,
and arose from the present study of the slant function.
In particular, ridges, sub-parabolic points, and binary differential
equations do not detect these two strata---the inflection and vertex of Gauss.
\end{rem}

\subsection{Proof of Theorem~\ref{thm:stratification}}
\label{subsec:proof-thm-strat}

\begin{proof}[Proof of Theorem~\ref{thm:stratification}]
The five cases of the theorem follow by combining
Proposition~\ref{prop:2jet-stratification} (which describes the
location and number of $A_{\ge 2}$-directions on the asymptotic normal
section) with Proposition~\ref{prop:type-normal} (which identifies
the precise type at the normal direction in cases~(2) and~(3)).
The detailed correspondence between the cases of the theorem and the
sub-cases of Proposition~\ref{prop:2jet-stratification} is stated
in Remark~\ref{rem:thm11-2jet-part}, and the precise types
$A_{3}$ (case~(2)) and $A_{5}$ (case~(3)) are identified by
Proposition~\ref{prop:type-normal}.
\end{proof}

\subsection{Proof of Theorem~\ref{thm:parabolic}}
\label{subsec:proof-thm-para}

\begin{proof}[Proof of Theorem~\ref{thm:parabolic}]
We address the four cases in turn, applying Corollary~\ref{cor:geom-alg}.
Throughout, $p$ is a non-flat parabolic point in Monge form under the
setting~\eqref{eq:para-conv}.

\medskip\noindent
\textbf{Case~(1): $p$ is a fold of Gauss but neither an inflection nor a vertex of Gauss.}
By Corollary~\ref{cor:geom-alg}~(1), $a_{30}\neq 0$,
$a_{30}a_{12}-a_{21}^{\,2}\neq 0$, and $\alpha_{2}\neq 0$.
Proposition~\ref{prop:type-normal}~(1) gives that the normal direction
$\mathbf{v}=N(p)$ is of type $A_{3}$.
Proposition~\ref{prop:type-second}~(1), with $\alpha_{2}\neq 0$, gives
that the second $A_{\ge 2}$-direction $\mathbf{v}_{\theta^{*}\!,0}$
is of type $A_{2}$.
Since $a_{30}a_{12}-a_{21}^{\,2}\neq 0$, equation~\eqref{eq:cosRule}
yields $\theta^{*}\neq\pi/2$
(cf.\ Remark~\ref{rem:alpha-collapse}), so
$\mathbf{v}_{\theta^{*}\!,0}$ is distinct from the asymptotic
direction.

\medskip\noindent
\textbf{Case~(2): $p$ is a first-order inflection of Gauss.}
By Corollary~\ref{cor:geom-alg}~(2), $a_{30}\neq 0$,
$a_{30}a_{12}-a_{21}^{\,2}=0$, and
$H_{4,u}(-a_{21},a_{30})\neq 0$.
Proposition~\ref{prop:type-normal}~(1) gives that the normal direction
is of type $A_{3}$.
Under $a_{30}a_{12}-a_{21}^{\,2}=0$, equation~\eqref{eq:cosRule}
reduces to $a_{30}\,a_{02}^{\,2}\cos\theta^{*}=0$, giving
$\theta^{*}=\pi/2$;
hence $\mathbf{v}_{\theta^{*}\!,0}=(1,0,0)$ is the asymptotic
direction itself
(Remark~\ref{rem:alpha-collapse}).
The collapsed form of $\alpha_{2}$ is
$a_{02}\,H_{4,u}(-a_{21},a_{30})$, which is nonzero by hypothesis.
Proposition~\ref{prop:type-second}~(1) then gives that the asymptotic
direction is of type $A_{2}$.

\medskip\noindent
\textbf{Case~(3): $p$ is a first-order vertex of Gauss.}
By Corollary~\ref{cor:geom-alg}~(3), $a_{30}\neq 0$,
$a_{30}a_{12}-a_{21}^{\,2}\neq 0$, $\alpha_{2}=0$, and
$\alpha_{3}\neq 0$.
Proposition~\ref{prop:type-normal}~(1) gives that the normal direction
is of type $A_{3}$.
Proposition~\ref{prop:type-second}~(2), with $\alpha_{2}=0$ and
$\alpha_{3}\neq 0$, gives that the second $A_{\ge 2}$-direction
$\mathbf{v}_{\theta^{*}\!,0}$ is of type $A_{3}$.

\medskip\noindent
\textbf{Case~(4): $p$ is a cusp of Gauss at which $\Sigma$ is regular.}
By Corollary~\ref{cor:geom-alg}~(4), $a_{30}=0$, $a_{21}\neq 0$, and
$a_{02}\,a_{40}-3\,a_{21}^{\,2}\neq 0$.
By Proposition~\ref{prop:2jet-stratification}~(2b-1), the unique
$A_{\ge 2}$-direction on the asymptotic normal section is the normal
direction $N(p)$.
Proposition~\ref{prop:type-normal}~(2), with
$a_{02}\,a_{40}-3\,a_{21}^{\,2}\neq 0$, gives that the normal
direction is of type $A_{5}$.
\end{proof}

\section*{Acknowledgements}

The authors are grateful to K. Anjyo for introducing them to this topic, and to T. Fukui and A. Hiramatsu for helpful discussions.
This work is partially supported by JSPS KAKENHI Grant Numbers 25H01485
and 25K00208.


\end{document}